\documentclass{amsart}
\usepackage{txfonts}
\usepackage{}

\newtheorem{theorem}{Theorem}[section]
\newtheorem{lemma}[theorem]{Lemma}
\newtheorem{corollary}[theorem]{Corollary}
\newtheorem{proposition}[theorem]{Proposition}

\theoremstyle{definition}
\newtheorem{definition}[theorem]{Definition}

\theoremstyle{remark}
\newtheorem{remark}[theorem]{Remark}

\numberwithin{equation}{section}

\begin{document}

\title{Einstein manifolds and  curvature operator of the second kind}

\author{Zhi-Lin Dai}

\address{Department of Mathematics,  Nanchang University, Nanchang, P. R. China 330031}

\author{Hai-Ping Fu}
\address{Department of Mathematics,  Nanchang University, Nanchang, P. R. China 330031}
\email{mathfu@126.com}
\thanks{Supported in part by National Natural Science Foundation of China \#12271069,  Jiangxi Province
Natural Science Foundation of China \#20202ACB201001.}


\date{}


\begin{abstract}
We prove that a compact Einstein manifold  of dimension $n\geq 4$ with nonnegative curvature operator of the second kind is a constant curvature space by Bochner technique. Moreover, we obtain that  compact Einstein manifolds  of dimension $n\geq 11$ with $\left [ \frac{n+2}{4}  \right ]$-nonnegative curvature operator of the second kind,  $4\ (\mbox{resp.},8,9,10)$-dimensional compact Einstein manifolds with $2$-nonnegative curvature of the  second kind and \linebreak   $5$-dimensional compact Einstein manifolds with $3$-nonnegative curvature of the second kind are   constant curvature spaces. Combing with Li's result \cite{L2}, we have  that a compact Einstein manifold  of dimension $n\geq 4$ with $\max\{4,\left [ \frac{n+2}{4}  \right ]\}$-nonnegative curvature operator of the second kind is a constant curvature space.
\end{abstract}

\maketitle

\section{Introduction}
In a Riemannian manifold, the Riemannian curvature tensor $R$ can induce two symmetric linear operators $\hat{R}$ and ${\mathop R\limits^ \circ  }$. Let $V$ be the tangent space of this Riemannian manifold,  $\wedge ^{2} \left ( V \right )$ denote the space of skewsymmetric $2$-tensor over $V$, and $S_0^2\left( V \right)$ denote the space of trace-free symmetric $2$-tensor over $V$.
The curvature operator of the first kind
$\hat{R}:\wedge ^{2} \left ( V \right ) \to  {\wedge ^{2} \left ( V \right )} $ is defined  by
$$\hat{R}\left ( e_{i}\wedge{e_{j} }  \right ) =\frac{1}{2} \sum_{k,l}^{} R_{ijkl} e_{k}\wedge{e_{l} },$$
and the curvature operator of the second kind
${\mathop R\limits^ \circ  } : S_0^2\left( V \right) \to S_0^2\left( V \right)$ is defined  in \cite{BK} by
$${\mathop R\limits^ \circ  } \left( {{e_i} \odot {e_j}} \right) = \sum\limits_{k,l} {{R_{iklj}}{e_k} \odot {e_l}}. $$
More details about the curvature operator of the second kind will be given below.

Studying Riemannian curvature operator of the second kind can be regarded as originating from the Nishikawa's conjecture in \cite {N} that a closed simply connected Riemannian manifold for which the curvature operator of the second kind is nonnegative and positive is diffeomorphic to a Riemannian locally symmetric space and a spherical space form, respectively.  Recently, Cao, Gursky and Tran \cite{CGT} proved the Nishikawa's conjecture in the case of $2$-positive curvature operator of the second kind. Later, Li \cite{L}  improved this result of Cao-Gursky-Tran to $3$-positive curvature operator of the second kind and proved the Nishikawa's conjecture in the case of three-nonnegative curvature operator of the second kind. Li also investigates how the curvature operator of the second kind affects the decomposability of manifolds. The relationship between the  $\frac{{3n}}{2}\frac{{n + 2}}{{n + 4}}$-nonnegative curvature operator of the second kind and the Betti numbers of Riemannian manifolds has been studied by using the Bochner technique in \cite{NPW} and \cite{NPWW}.

 Kashiwada \cite{K}  has shown that Riemannian manifolds with harmonic curvature  for which the curvature operator of the second kind are nonnegative and positive are locally symmetric spaces and constant curvature spaces, respectively.  Cao, Gursky and Tran \cite{CGT} have  shown that Einstein manifolds of dimension $n\geq4$ with  four-positive curvature operator of the second kind and four-nonnegative curvature operator of the second kind are constant curvature spaces and locally symmetric spaces, respectively, which  were then weakened to $4\frac12$-positive curvature operator of the second kind and $4\frac12$-nonnegative curvature operator of the second kind by Li in \cite{L2}. First, they proved that  Einstein manifolds of dimension $n\geq4$ with  $4\frac12$-positive curvature operator of the second kind and $4\frac12$-nonnegative curvature operator of the second kind satisfy PIC and nonnegative isotropic curvature, respectively. Second, by using Brendle's result in \cite{Br} that Einstein manifolds of dimension $n\geq4$ with PIC (resp., nonnegative isotropic curvature) are constant curvature spaces (resp., locally symmetric spaces), they obtained the above result. Nienhaus, Petersen and Wink \cite{NPW} show that  $n$-dimensional compact Einstein manifolds with $k(<\frac{3n(n+2)}{2(n+4)})$-nonnegative
curvature operators of the second kind are either rational homology spheres or flat. We refer to \cite{{CMR},{PW}} for details on compact Riemannian manifolds with $k$-positive curvature operator of the first kind by using Bochner technique.

  In this paper, we investigate Einstein manifolds of dimension $n\geq4$ with nonnegative curvature operator of the second kind, and obtain the following Theorems 1.1 and 1.2.

 \begin{theorem}
\label{th1.1}
Let $\left( {M,g} \right)$ be a  compact Einstein manifold of dimension $n\geq 6$, and ${\mathop R\limits^ \circ  }$ be the curvature operator of the second kind. If one of the following conditions holds:

(1) $6\leq n\leq7$ and ${\mathop R\limits^ \circ  }$ is nonnegative;

(2) $8\leq n\leq 10$ and ${\mathop R\limits^ \circ  }$ is $2$-nonnegative;

(3) $n\geq 11$ and ${\mathop R\limits^ \circ  }$ is $ [\frac {n+2}{4}]$-nonnegative,

\noindent
 then $\left( {M,g} \right)$ is a constant curvature space.
\end{theorem}

\begin{theorem}
\label{th1.3}
Let $\left( {M,g} \right)$ be an $n(n=4,5)$-dimensional  compact Einstein manifold, ${\mathop R\limits^ \circ  }$ be the curvature operator of the second kind. Then

(1) if $n=4$ and ${\mathop R\limits^ \circ  }$ is $2$-nonnegative, then $(M,g)$ is a constant space;

(2) if $n=5$ and ${\mathop R\limits^ \circ  }$ is $3$-nonnegative, then $(M,g)$ is a constant space.
\end{theorem}

\begin{corollary}
Let $\left( {M,g} \right)$ be a  compact Einstein manifold of dimension $n\geq 4$ with nonnegative curvature operator of the second kind. Then $\left( {M,g} \right)$ is a constant curvature space.
\end{corollary}

\begin{remark}
For compact Einstein manifolds of dimension $n\geq 4$ with nonnegative curvature operator of the second kind, the conclusion that has been proved before is that they are Riemannian locally symmetric spaces. We prove these theorems only using Bochner technique without using Ricci flow. In \cite{DF}, the authors should prove that  a compact manifold with harmonic Weyl tensor and nonnegative curvature operator of the second kind is  globally conformally equivalent to a space of positive constant curvature or is isometric to a  flat manifold, by using Corollary 1.3.
\end{remark}

It is easy to see from d) of Theorem 3.6 in \cite{NPW} that the curvature operator of the second kind which is $4$-nonnegative  is either $4\frac12$-positive or nonnegative. Combing Theorem 1.1 and Corollary 1.3 with Theorem 1.9 in \cite{L2}, we have the following:
\begin{corollary}
Let $\left( {M,g} \right)$ be a  compact Einstein manifold of dimension $n\geq 4$ with $\max\{4,\left [ \frac{n+2}{4}  \right ]\}$-nonnegative curvature operator of the second kind. Then $\left( {M,g} \right)$ is a constant curvature space.
\end{corollary}

{Acknowledgement:} The authors are very grateful to the referee for his/her helpful suggestions. The second author also sincerely thanks Dr. Xiaolong Li for his helpful suggestions.


\section{Preliminaries}
Let $\left( {V,g} \right)$ be an $n$-dimensional Euclidean vector space and $\left\{ {{e_i}} \right\}_{i = 1}^n$ be an orthonormal basis for $V$. We recall that the second-order tensor space can be decomposed by
$${T^2}\left( V \right) = { \wedge ^2}\left( V \right) \oplus {S^2}\left( V \right),$$ where ${ \wedge ^2}\left( V \right)$ and ${S^2}\left( V \right)$ is the second-order antisymmetric tensor space and the second-order symmetric tensor space on $V$, respectively. Every element of ${S^2}\left( V \right)$ can  be seen as a self-adjoint linear operator on vector space $V$. ${S^2}\left( V \right)$ is not irreducible under the action of the orthogonal group on $V$.
Let $S_0^2\left( V \right)$ denote the space of the second-order traceless symmetric tensor space on $V$. Thus  ${S^2}\left( V \right)$ can be decomposed into irreducible subspaces by
$${S^2}\left( V \right) = S_0^2\left( V \right) \oplus \mathbb{R}g,$$
where $g = \sum\limits_{i = 1}^n {{e_i} \otimes {e_i}} $ is the identity on $V$.

Let $V$ be  the tangent  space  of an $n$-dimensional Riemannian manifold $\left( {M,g} \right)$ at  $p\in M$. Thus the $\left( {0,4} \right)$-Riemannian curvature tensor $R$ is defined as $$R\left( {X,Y} \right)Z =  - {\nabla _X}{\nabla _Y}Z + {\nabla _Y}{\nabla _X}Z + {\nabla _{\left[ {X,Y} \right]}}Z,$$ $$R\left( {X,Y,Z,W} \right) = \left\langle {R\left( {X,Y} \right)Z,W} \right\rangle ,$$ i.e. ${R_{ijkl}} = R\left( {{e_i},{e_j},{e_k},{e_l}} \right)$. Also ${R_{jl}} = \sum\limits_i {{R_{ijil}}} $ and $s = \sum\limits_j {{R_{jj}}} $ is the components of Ricci curvature tensor and scalar curvature tensor, respectively.

The curvature operator $\widetilde R$ act on ${S^2}\left( V \right)$ induced by Riemannian curvature tensor $R$ is defined by
$$\widetilde R ({e_i} \odot {e_j}) = \sum\limits_{k,l} {{R_{kijl}}{e_k} \odot {e_l}}, $$
where ${e_i} \odot {e_j} = {e_i} \otimes {e_j} + {e_j} \otimes {e_i}$. So ${\left\{ {\frac{1}{{\sqrt 2 }}{e_i} \odot {e_j}} \right\}_{i < j}} \cup \left\{ {\frac{1}{2}{e_i} \odot {e_i}} \right\}, i, j=1,\cdots,n,$ is an orthonormal basis for ${S^2}\left( V \right)$. $\widetilde R$ induces a symmetric form ${\mathop R\limits^ \circ } $  defined by
$${\mathop R\limits^ \circ} :S_0^2\left( V \right) \times S_0^2\left( V \right) \to \mathbb{R},$$
$$ \left( {{S^\alpha },{S^\beta }} \right) \mapsto {R_{ijkl}}S_{jk}^\alpha S_{il}^\beta.$$
It can be regarded as $\widetilde R$ and  its image restrict on $S_0^2\left( V \right)$.
The operator ${\mathop R\limits^ \circ}$ is called  curvature operator of the second kind.

Let $\left\{ {{\lambda _\alpha }} \right\}_{\alpha=1}^{\frac{(n-1)(n+2)}{2}}$ be the eigenvalues of ${\mathop R\limits^ \circ}$. We always assume $\left\{ {{\lambda _\alpha }} \right\}$ be a nondecreasing sequence in this paper. For some number $\delta\geq1$,  we say the  curvature operator of the second kind ${\mathop R\limits^ \circ}$ is $\delta$-nonnegative (positive), if $\sum\limits_{\alpha {\rm{ = }}1}^{[\delta]} {{\lambda _\alpha }} +(\delta-[\delta])\lambda _{[\delta]+1} \ge 0$ $( >0 )$. If $\delta=k$ is positive integer, then ${\mathop R\limits^ \circ}$ is $k$-nonnegative (positive), i.e., $\sum\limits_{\alpha {\rm{ = }}1}^{k} {{\lambda _\alpha }} \ge 0$ $( >0 )$.

\begin{definition}[\cite{NPW}]
\label{def2.1}
Let $(V,g)$ be an $n$-dimensional Euclidean vector space, $T^{(0,k)}(V)$ is $(0,k)$-tensor space. For $S \in {S^2}\left( V \right)$ and $T \in {T^{(0,k)}}\left( V \right)$,  define the self-adjoint linear operator $S$ as
$$S:{T^{(0,k)}}(V) \to {T^{(0,k)}}(V),$$ $$(ST)({X_1},\cdots,{X_k}) = \sum\limits_{i = 1}^k {T({X_1},\cdots ,S{X_i},\cdots,{X_k})} ,$$
and ${T^{{S^2}}} \in {T^{(0,k)}}(V) \otimes {S^2}(V)$ via
$$\langle {T^{{S^2}}}({X_1},\cdots, {X_k}),S\rangle  = (ST)({X_1},\cdots,{X_k}),$$ i.e., for an orthonormal basis $\left\{ {{\widetilde S^\alpha }} \right\}$ of $S^2\left( V \right)$,
$${T^{{S^2}}} = \sum\limits_\alpha  {{\widetilde S^\alpha }T \otimes {\widetilde S^\alpha }} .$$ Similarly, define ${T^{{S_0^2}}} \in {T^{(0,k)}}(V) \otimes {S_0^2}(V)$ and for an orthonormal basis $\left\{ {{S^\alpha }} \right\}$ of $S_0^2\left( V \right)$, $${T^{S_0^2}} = \sum\limits_\alpha  {{S^\alpha }T \otimes {S^\alpha }} .$$
\end{definition}

By Definition 2.1, we have
$${T^{{S^2}}}={T^{S_0^2}} + \frac{g}{{\sqrt n }}T \otimes \frac{g}{{\sqrt n }} = {T^{S_0^2}} + \frac{k}{n}T \otimes g.$$
For a tensor ${T^{{S^2}}} \in {T^{(0,k)}}(V) \otimes {S^2}(V)$, we have
\begin{equation}
\label{2.1}
\begin{aligned}
T^{{S^2}}
&= \sum\limits_{i < j} {\frac{1}{{\sqrt 2 }}\left( {{e_i} \odot {e_j}} \right)T \otimes \frac{1}{{\sqrt 2 }}\left( {{e_i} \odot {e_j}} \right)}  + \sum\limits_i {\frac{1}{2}\left( {{e_i} \odot {e_i}} \right)T \otimes \frac{1}{2}\left( {{e_i} \odot {e_i}} \right)} \\
&=\frac{1}{4}\sum\limits_{i \ne j} {\left( {{e_i} \odot {e_j}} \right)T \otimes \left( {{e_i} \odot {e_j}} \right)}  + \frac{1}{4}\sum\limits_i {\left( {{e_i} \odot {e_i}} \right)T \otimes \left( {{e_i} \odot {e_i}} \right)}  \\
&=\frac{1}{4}\sum\limits_{i,j} {\left( {{e_i} \odot {e_j}} \right)T \otimes \left( {{e_i} \odot {e_j}} \right)} .
\end{aligned}
\end{equation}

Let $\left\{ {{\widetilde S^\alpha }} \right\}$ be the orthonormal eigenbasis of curvature operator $\widetilde R$ with the corresponding eigenvalues  $\left\{ {{\widetilde \lambda _\alpha }} \right\}$. Set $\widetilde R$ acting on $T^{{S^2}}$ as
$$\widetilde R \left({T^{{S^2}}} \right) = \sum\limits_\alpha  {{\widetilde S^\alpha }} T \otimes \widetilde R \left({\widetilde S^\alpha }\right),$$ and
\begin{equation}
\label{2.2}
\left\langle {\widetilde R \left({T^{{S^2}}}\right),{T^{{S^2}}}} \right\rangle  = {\sum\limits_\alpha  {{\widetilde \lambda _\alpha }\left| {{\widetilde S^\alpha }T} \right|} ^2}.\end{equation}

\begin{remark}
The above content mainly comes from Preliminaries in \cite{NPW}.
\end{remark}

\section{A Bochner formula for curvature operator of the second kind on Einstein manifolds}
In a Riemannian manifold, the Riemannian curvature tensor can be decomposed into irreducible components  \cite{B} as
$$R = W + \frac{1}{{n - 2}}E \circledwedge g + \frac{s}{{2n\left( {n - 1} \right)}}g \circledwedge g,$$
$${R_{ijkl}} = {W_{ijkl}}{\rm{ + }}\frac{1}{{n - 2}}\left( {{E_{ik}}{g_{jl}} + {E_{jl}}{g_{ik}} - {E_{il}}{g_{jk}} - {E_{jk}}{g_{il}}} \right) + \frac{s}{{n\left( {n - 1} \right)}}\left( {{g_{ik}}{g_{jl}} - {g_{il}}{g_{jk}}} \right).$$
Here $A\circledwedge B$ is the Kulkarni-Nomizu product of  second-order symmetric tensors $A$ and $B$.  $W$ is the Weyl curvature tensor and $E=R-\frac{s}{n}g$ is the traceless Ricci tensor. A Riemannian manifold is an Einstein manifold if and only if $E=0$.

In this section, we will prove some Lemmas and Propositions for Einstein manifolds. And some of them are also true for Riemannian manifolds with harmonic curvature.

Let $T$ be  an smooth algebraic curvature tensor. To facilitate the calculation of some quantities, we set
$${R_{tpsq}}{T_{tpkl}}{T_{sqkl}} = \alpha ,{R_{tspq}}{T_{tjpl}}{T_{sjql}} = \beta .$$
By using the first Bianchi identity, we have
$${R_{tspq}}{T_{tpkl}}{T_{sqkl}} = \frac{1}{2}\alpha,$$
since $${R_{tspq}}{T_{tpkl}}{T_{sqkl}} = \left( {{R_{tpsq}} + {R_{tqps}}} \right){T_{tpkl}}{T_{sqkl}} = \alpha  - {R_{tspq}}{T_{tpkl}}{T_{sqkl}},$$
and
$${R_{tpsq}}{T_{tkpl}}{T_{skql}}= \frac{1}{4}\alpha,$$
since
\begin{equation*}
\begin{aligned}
{R_{tpsq}}{T_{tpkl}}{T_{sqkl}}
=&{R_{tpsq}}\left({T_{tkpl}} + {T_{tlkp}}\right)\left({T_{skql}} + {T_{slkq}}\right)\\
=& {R_{tpsq}}{T_{tkpl}}{T_{skql}} + {R_{tpsq}}{T_{tkpl}}{T_{slkq}} + {R_{tpsq}}{T_{tlkp}}{T_{skql}} + {R_{tpsq}}{T_{tlkp}}{T_{slkq}}\\
=& 2{R_{tpsq}}{T_{tkpl}}{T_{skql}} - 2{R_{tpsq}}{T_{tkpl}}{T_{slqk}}\\
=& 2{R_{tpsq}}{T_{tkpl}}{T_{skql}} + 2{R_{tpqs}}{T_{tkpl}}{T_{qksl}}\\
=& 4{R_{tpsq}}{T_{tkpl}}{T_{skql}},\\
\end{aligned}
\end{equation*}
and
$${R_{pstq}}{T_{tjpl}}{T_{sjql}} = \left( {{R_{ptsq}} + {R_{pqts}}} \right){T_{tjpl}}{T_{sjql}} =  - {R_{tpsq}}{T_{tjpl}}{T_{sjql}} + \beta  =  - \frac{1}{4}\alpha  + \beta .$$

Based on the above quantities, we can get the following propositions.

\begin{proposition}
\label{pro3.1}
Let $T$ be a smooth algebraic curvature tensor on a Riemannian manifold $(M,g)$ of dimension $n\geq3$. If $T$ is harmonic $($ i.e. $T$ satisfies the second Bianchi identity and $\sum\limits_l {{T_{ijkl,l}} = 0}$ $)$, then
\begin{equation}
\label{3.1}
\begin{aligned}
\left\langle {\Delta T,T} \right\rangle
=& 2{R_{lt}}{T_{ijkl}}{T_{ijkt}}  - {R_{tskl}}{T_{ijkl}}{T_{ijts}}- 4{R_{itls}}{T_{ijlk}}{T_{tjsk}} \\
=& 2{R_{lt}}{T_{ijkl}}{T_{ijkt}} -\alpha  -4 \beta.
\end{aligned}
\end{equation}
\end{proposition}
\begin{proof}By using the second Bianchi identity and Ricci identity, we have
\begin{equation*}
\begin{aligned}
  \left\langle {\Delta T,T} \right\rangle&={T_{ijkl,ss}}{T_{ijkl}} = \left( {{T_{ijks,ls}} + {T_{ijsl,ks}}} \right){T_{ijkl}} = 2{T_{ijks,ls}}{T_{ijkl}}  \\
  &=2{T_{ijkl}}\left( {{T_{ijks,sl}} + {R_{lsti}}{T_{tjks}} + {R_{lstj}}{T_{itks}} + {R_{lstk}}{T_{ijts}} + {R_{lsts}}{T_{ijkt}}} \right)  \\
  & = 4{R_{lsti}}{T_{ijkl}}{T_{tjks}} + 2{R_{lstk}}{T_{ijkl}}{T_{ijts}} + 2{R_{lt}}{T_{ijkl}}{T_{ijkt}}\\
  & = 2{R_{lt}}{T_{ijkl}}{T_{ijkt}}  - {R_{tskl}}{T_{ijkl}}{T_{ijts}}- 4{R_{itls}}{T_{ijlk}}{T_{tjsk}}\\
  &= 2{R_{lt}}{T_{ijkl}}{T_{ijkt}} -\alpha  -4 \beta.
\end{aligned}
\end{equation*}
The proof is completed.
\end{proof}

\begin{proposition}
\label{pro3.2}
Let $T$ and ${\mathop R\limits^ \circ}$ denote a smooth algebraic curvature tensor and a curvature operator of the second kind on a Riemannian manifold $(M,g)$ of dimension $n\geq3$, respectively. Then
\begin{equation}
\label{3.2}
\begin{aligned}
\left\langle {{\mathop R\limits^ \circ} \left({T^{S_0^2}}\right),{T^{S_0^2}}} \right\rangle
  =&\frac{{2n + 32}}{n}{R_{st}}{T_{sjkl}}{T_{tjkl}} - 5{R_{tpsq}}{T_{tpkl}}{T_{sqkl}}\\
   &+ 4{R_{tspq}}{T_{tjpl}}{T_{sjql}} - \frac{{16}}{{{n^2}}}s{\left| T \right|^2}\\
  =&\frac{{2n + 32}}{n}{R_{st}}{T_{sjkl}}{T_{tjkl}} - 5\alpha +4\beta - \frac{{16}}{{{n^2}}}s{\left| T \right|^2}.
\end{aligned}
\end{equation}
\end{proposition}
\begin{proof}
First, by using \eqref{2.1} we obtain
\begin{equation*}
\begin{aligned}
\left\langle {\widetilde R\left( {{T^{{S^2}}}}\right),{T^{{S^2}}}} \right\rangle
=&\frac{1}{{16}}\sum\limits_{s,t,p,q} {\left\langle {\left( {{e^s} \odot {e^t}}\right)T \otimes \widetilde R \left( {{e^s} \odot {e^t}} \right),\left( {{e^p} \odot {e^q}}\right)T \otimes\left( {{e^p} \odot {e^q}}\right)} \right\rangle } \\
  =&\frac{1}{8}\left( {{R_{pstq}} + {R_{qstp}}} \right)\left( {{\delta _{is}}{T_{tjkl}} + \cdots + {\delta _{ls}}{T_{ijkt}} + {\delta _{it}}{T_{sjkl}} + \cdots + {T_{ijks}}{\delta _{lt}}} \right)\\
  &\times\left( {{\delta _{ip}}{T_{qjkl}} + \cdots + {\delta _{lp}}{T_{ijkq}} + {\delta _{iq}}{T_{pjkl}} + \cdots + {\delta _{lq}}{T_{ijkp}}} \right)\\
  =&\left( {{R_{pstq}} + {R_{qstp}}} \right) {\delta _{is}}T_{tjkl}
   \left( {{\delta _{ip}}{T_{qjkl}} + \cdots + {\delta _{lp}}{T_{ijkq}} + {\delta _{iq}}{T_{pjkl}} + \cdots + {\delta _{lq}}{T_{ijkp}}} \right)\\
  =&\left({{T_{tpkl}}{T_{sqkl}} + {T_{tjpl}}{T_{sjql}}}
   +{{T_{tjkp}}{T_{sjkq}} + {T_{tqkl}}{T_{spkl}} + {T_{tjql}}{T_{sjpl}} + {T_{tjkq}}{T_{sjkp}}} \right)\\
  &\times\left( {{R_{pstq}} + {R_{ptsq}}} \right)+ 2{R_{ps}}{T_{pjkl}}{T_{sjkl}}\\
  =&{\rm{2}}{R_{st}}{T_{sjkl}}{T_{tjkl}} + 2\left( {{R_{pstq}} - {R_{tpsq}}} \right)\left( {{T_{tpkl}}{T_{sqkl}} + 2{T_{tjpl}}{T_{sjql}}} \right)\\
  =&{\rm{2}}{R_{st}}{T_{sjkl}}{T_{tjkl}} + 2\left( -{R_{tspq}}{T_{tpkl}}{T_{sqkl}} - {R_{tpsq}}{T_{tpkl}}{T_{sqkl}}
   + 2{R_{pstq}}{T_{tjpl}}{T_{sjql}} - 2{R_{tpsq}}{T_{tjpl}}{T_{sjql}} \right) \\
  =&2{R_{st}}{T_{sjkl}}{T_{tjkl}} + 2\left( { - \frac{1}{2}\alpha  - \alpha  + 2\left( { - \frac{1}{4}\alpha  + \beta } \right) -  {\frac{1}{2}\alpha } } \right)\\
  =&2{R_{st}}{T_{sjkl}}{T_{tjkl}} - 5\alpha  + 4\beta .
\end{aligned}
\end{equation*}
Second, by ${T^{{S^2}}} = {T^{S_0^2}} + \frac{k}{n}T \otimes g$, we have from the above
\begin{equation*}
\begin{aligned}
\left\langle {{\mathop R\limits^ \circ}\left ({T^{S_0^2}}\right),{T^{S_0^2}}} \right\rangle
=&\left\langle {\widetilde R \left({T^{{S^2}}} - \frac{4}{n}T \otimes g\right),{T^{{S^2}}} - \frac{4}{n}T \otimes g} \right\rangle\\
=&\left\langle {\widetilde R \left({T^{{S^2}}}\right),{T^{{S^2}}}} \right\rangle  + \frac{8}{n}\left\langle {{T^{{S^2}}},T \otimes Ric} \right\rangle  - \frac{{16}}{{{n^2}}}s{\left| T \right|^2}\\
=&\left\langle {\widetilde R \left({T^{{S^2}}}\right),{T^{{S^2}}}} \right\rangle  + \frac{2}{n}\sum\limits_{s,t} {\left\langle {\left({e^s} \odot {e^t}\right)T \otimes \left({e^s} \odot {e^t}\right),T \otimes Ric} \right\rangle }  - \frac{{16}}{{{n^2}}}s{\left| T \right|^2}\\
=&\left\langle {\widetilde R \left({T^{{S^2}}}\right),{T^{{S^2}}}} \right\rangle  + \frac{2}{n}\sum\limits_{s,t} {\left\langle {\left({e^s} \odot {e^t}\right)T,T} \right\rangle \left\langle {{e^s} \odot {e^t},Ric} \right\rangle }  - \frac{{16}}{{{n^2}}}s{\left| T \right|^2}\\
=&\left\langle {\widetilde R \left({T^{{S^2}}}\right),{T^{{S^2}}}} \right\rangle  + \frac{2}{n}\sum\limits_{s,t} {8{T_{sjkl}}{T_{tjkl}}\left( {{R_{st}} + {R_{ts}}} \right)}   - \frac{{16}}{{{n^2}}}s{\left| T \right|^2}\\
 =&\frac{{2n + 32}}{n}{R_{st}}{T_{sjkl}}{T_{tjkl}} - 5\alpha +4\beta - \frac{{16}}{{{n^2}}}s{\left| T \right|^2}.
\end{aligned}
\end{equation*}
The proof is completed.
\end{proof}

\begin{lemma}
\label{pro3.3}
Let $(M,g)$ be an $n$-dimensional Einstein manifold. Then the  curvature operator of the second kind ${\mathop R\limits^ \circ} $ on $M$ has
$tr{\mathop R\limits^ \circ} = \frac{{n + 2}}{{2n}}s,$ and \begin{equation*}
\begin{aligned}
tr(({\mathop R\limits^ {\circ}})^2) =& \frac{3}{4}{\left| R \right|^2} - \frac{s^2}{n^2},\\
tr(({\mathop R\limits^ \circ})^3) =& -\beta+\frac18\alpha+\frac{s^3}{n^3}.
\end{aligned}
\end{equation*}
\end{lemma}
\begin{proof}
For a self-adjoint operator $A$ on $S^2(V)$, we have
\begin{equation*}
\begin{aligned}
trA
 =& \sum\limits_{i < j} {\left\langle {A(\frac{1}{{\sqrt 2 }}{e_i} \odot {e_j}),\frac{1}{{\sqrt 2 }}{e_i} \odot {e_j}} \right\rangle }  + \sum\limits_i {\left\langle {A(\frac{1}{2}{e_i} \odot {e_i}),\frac{1}{2}{e_i} \odot {e_i}} \right\rangle }\\
 =& \frac{1}{4}\sum\limits_{i \ne j} {\left\langle {A({e_i} \odot {e_j}),{e_i} \odot {e_j}} \right\rangle }  + \frac{1}{4}\sum\limits_i {\left\langle {A({e_i} \odot {e_i}),{e_i} \odot {e_i}} \right\rangle } \\
 =& \frac{1}{4}\sum\limits_{i,j} {\left\langle {A({e_i} \odot {e_j}),{e_i} \odot {e_j}} \right\rangle }\\
 =& \frac{1}{2}\sum\limits_{i,j} {{A_{ijij}}},
\end{aligned}
\end{equation*}
and
$$\left\langle {A(\frac{g}{{\sqrt n }}),\frac{g}{{\sqrt n }}} \right\rangle  = \frac{1}{{4n}}\sum\limits_{i,j} {\left\langle {A({e_i} \odot {e_i}),{e_j} \odot {e_j}} \right\rangle }  = -\frac{1}{n}\sum\limits_{i,j} {{A_{ijij}}}.$$
Thus first we have
$$tr{\mathop R\limits^ \circ} = tr\widetilde R  - \left\langle {\widetilde R ( {\frac{g}{{\sqrt n }}} ),\frac{g}{{\sqrt n }}} \right\rangle  = \frac{{n + 2}}{{2n}}s.$$
Next, we get
\begin{equation*}
\begin{aligned}
tr({\widetilde R ^2}) =& \frac{1}{4}\sum\limits_{i,j} {\left\langle {{{\widetilde R }^2}({e_i} \odot {e_j}),{e_i} \odot {e_j}} \right\rangle }\\
 =& \frac{1}{4}\sum\limits_{i,j,k,l} {\left\langle {\widetilde R ({R_{kijl}}({e_k} \odot {e_l})),{e_i} \odot {e_j}} \right\rangle }  \\
 =& \frac{1}{4}\sum\limits_{i,j,k,l,s,t} {\left\langle {{R_{sklt}}{R_{kijl}}({e_s} \odot {e_t}),{e_i} \odot {e_j}} \right\rangle } \\
 =& \frac{1}{2}\sum\limits_{i,j,k,l} {({R_{iklj}} + {R_{jkli}}){R_{kijl}}} \\
 =& \frac{3}{4}{\left| R \right|^2},\\
\end{aligned}
\end{equation*}
and
\begin{equation*}
\begin{aligned}
\left\langle {{{\widetilde R }^2}(\frac{g}{{\sqrt n }}),\frac{g}{{\sqrt n }}} \right\rangle =& \left\langle {{{\widetilde R }^2}(\frac{1}{{2\sqrt n }}\sum\limits_i {({e_i} \odot {e_i})} ),\frac{1}{{2\sqrt n }}\sum\limits_j {({e_j} \odot {e_j})} } \right\rangle\\
=&\frac{1}{{4n}}\sum\limits_{i,j,k,l,s,t} {\left\langle {{R_{sklt}}{R_{kiil}}{e_s} \odot {e_t},{e_j} \odot {e_j}} \right\rangle } \\
=&\frac{1}{n}{\left| {Ric} \right|^2}\\
=&\frac{s^2}{n^2}.
\end{aligned}
\end{equation*}
For Einstein manifolds, the $S_0^2(V)$ is the invariant subspace of $\widetilde R$, and we have
$$tr(({\mathop R\limits^ {\circ}})^2) = tr({\widetilde R ^2}) - \left\langle {{{\widetilde R }^2}(\frac{g}{{\sqrt n }}),\frac{g}{{\sqrt n }}} \right\rangle =\frac{3}{4}{\left| R \right|^2} - \frac{s^2}{n^2}.$$
Similarly we get
$$tr({\widetilde R ^3}) =  - {R_{isjt}}{R_{kslt}}{R_{ikjl}} + \frac{1}{8}{R_{ijkl}}{R_{ijst}}{R_{klst}},$$
$$\left\langle {{{\widetilde R }^3}(\frac{g}{{\sqrt n }}),\frac{g}{{\sqrt n }}} \right\rangle=\frac{1}{n}{R_{st}}{R_{kl}}{R_{sklt}}= -\frac{s^3}{n^3} ,$$
\begin{equation*}
\begin{aligned}tr(({\mathop R\limits^ \circ})^3) =& - {R_{isjt}}{R_{kslt}}{R_{ikjl}} + \frac{1}{8}{R_{ijkl}}{R_{ijst}}{R_{klst}}  +\frac{s^3}{n^3}\\
=&-\beta+\frac18\alpha+\frac{s^3}{n^3}.\end{aligned}
\end{equation*}
The proof is completed.
\end{proof}

\begin{proposition}
\label{lem3.4}
Let $(M,g)$ be an $n$-dimensional Einstein manifold. If $\left\{ {{S^\alpha }} \right\}$ is the orthonormal eigenbasis for the curvature operator of the second kind ${\mathop R\limits^ \circ}$ on $(M,g)$ with the corresponding eigenvalues $\left\{ {{\lambda _\alpha }} \right\}$, then
\begin{equation}
\label{3.3}
\begin{aligned}
3\langle \Delta R,R\rangle
  =&\sum {{\lambda _\alpha }{{\left| {{S^\alpha }W} \right|}^2}}+8\left(\frac{-n^3+6n^2+12n-8}{3n^4(n-1)^2} \right){s^3}\\
   &+8\left(\frac{2n^2-22n+8}{3n^2(n-1)} \right)s\sum {\lambda _\alpha ^2}  + 16\sum {\lambda _\alpha ^3}.
\end{aligned}
\end{equation}
\end{proposition}
\begin{proof}
By Prositions \ref{pro3.1}  and \eqref{3.2}, and Lemma \ref{pro3.3}, we have
\begin{equation*}
\begin{aligned}\left\langle {\Delta R,R} \right\rangle  = \frac{2}{n}s{\left| R \right|^2} - \frac{3}{2}\alpha  + 4tr\left( {({\mathop {{R}}\limits^ \circ})^3  } \right) - \frac{4}{{{n^3}}}{s^3},\\
\left\langle {{\mathop R\limits^ \circ} ({R^{S_0^2}}),{R^{S_0^2}}} \right\rangle = \frac{{2n + 16}}{{{n^2}}}s{\left| R \right|^2} - \frac{9}{2}\alpha  - 4tr\left( {{({\mathop R\limits^ \circ  })^3}} \right) + \frac{4}{{{n^3}}}{s^3}.\end{aligned}
\end{equation*}
Combining the above two equations, we get
\begin{equation}3\left\langle {\Delta R,R} \right\rangle  = \left\langle {{\mathop R\limits^ \circ} ({R^{S_0^2}}),{R^{S_0^2}}} \right\rangle  + \frac{{4n - 16}}{{{n^2}}}s{\left| R \right|^2} + 16\sum {{\lambda _\alpha^3 }}  - \frac{{16}}{{{n^3}}}{s^3}.\end{equation}

 ${\mathop W\limits^ \circ}$ is the curvature operator induced by Weyl curvature tensor $W$, acting on $S_0^2(V)$ just like ${\mathop R\limits^ \circ}$. Since $(M,g)$ is an Einstein manifold, we choose $\left\{ {{S^\alpha }} \right\}$ which is an orthonormal eigenbasis for ${\mathop R\limits^ \circ}$  with  corresponding eigenvalues  $\left\{ \lambda _\alpha\right\}$. Thus $\left\{ {{S^\alpha }} \right\}$ is also orthonormal eigenbasis for ${\mathop W\limits^ \circ}$ with  corresponding eigenvalues  $\left\{ \omega _\alpha\right\}$ satisfying $\omega _\alpha=\lambda _\alpha-\frac{s}{n(n-1)}$ and $\sum\limits_\alpha  {{\omega _\alpha }}  = 0$.
Now we compute
\begin{equation}
\begin{aligned}
 \langle {{\mathop R\limits^ \circ}}  ({R^{S_0^2}}),{R^{S_0^2}}\rangle
  =&\sum {{\lambda _\alpha }{{\left| {{S^\alpha }R} \right|}^2}}  =\sum {{\lambda _\alpha }{{\left| {{S^\alpha }(W + \frac{s}{{2n(n - 1)}}g \circledwedge g)} \right|}^2}}\\
  =&\sum {{\lambda _\alpha }{{\left| {{S^\alpha }W} \right|}^2}}  +  \frac{{{s^2}}}{{4{n^2}{{(n - 1)}^2}}}\sum {{\lambda _\alpha }}{\left| {{S^\alpha }(g \circledwedge g)} \right|^2}\\
   &+ \frac{s}{{n(n - 1)}}\sum {{\lambda _\alpha }\left\langle {{S^\alpha }W,{S^\alpha }(g \circledwedge g)} \right\rangle }\\
   =&\sum {{\lambda _\alpha }{{\left| {{S^\alpha }W} \right|}^2}}  +  \frac{{{4s^2}}}{{{n^2}{{(n - 1)}^2}}}\sum {{\lambda _\alpha }}{\left| {{S^\alpha }\circledwedge g} \right|^2}\\
   &+ \frac{s}{{n(n - 1)}}\sum {{\lambda _\alpha }\left\langle {W,{S^\alpha }({S^\alpha }(g \circledwedge g))} \right\rangle }\\
    =&\sum {{\lambda _\alpha }{{\left| {{S^\alpha }W} \right|}^2}}  +  \frac{{{16s^2}}}{{{n^2}{{(n - 1)}^2}}}\sum {{\lambda _\alpha }}((n - 2){\left| {{S^\alpha }} \right|^2} + {(tr{S^\alpha })^2})\\
   &+ \frac{8s}{{n(n - 1)}}\sum {\lambda _\alpha }\langle W,({S^\alpha } \circ {S^\alpha }) \circledwedge g + {S^\alpha } \circledwedge {S^\alpha }\rangle\\
   =&\sum {{\lambda _\alpha }{{\left| {{S^\alpha }W} \right|}^2}}  +  \frac{{{16(n-2)s^2}}}{{{n^2}{{(n - 1)}^2}}}\sum {{\lambda _\alpha }}
   + \frac{8s}{{n(n - 1)}}\sum {\lambda _\alpha }\langle W, {S^\alpha } \circledwedge {S^\alpha }\rangle\\
   =&\sum {{\lambda _\alpha }{{\left| {{S^\alpha }W} \right|}^2}}  +  \frac{{{16(n-2)s^2}}}{{{n^2}{{(n - 1)}^2}}}\sum {{\lambda _\alpha }}
   + \frac{8s}{{n(n - 1)}}\sum {\lambda _\alpha }{W_{ijkl}}{({S^\alpha } \circledwedge {S^\alpha })_{ijkl}}\\
   =&\sum {{\lambda _\alpha }{{\left| {{S^\alpha }W} \right|}^2}}  +  \frac{{{16(n-2)s^2}}}{{{n^2}{{(n - 1)}^2}}}\sum {{\lambda _\alpha }}
   + \frac{16s}{{n(n - 1)}}\sum {\lambda _\alpha }{W_{ijkl}}({S_{ik}^\alpha }{S_{jl}^\alpha } - {S}_{il}^\alpha {S_{jk}^\alpha })\\
   =&\sum {{\lambda _\alpha }{{\left| {{S^\alpha }W} \right|}^2}}  +  \frac{{{16(n-2)s^2}}}{{{n^2}{{(n - 1)}^2}}}\sum {{\lambda _\alpha }}
   + \frac{16s}{{n(n - 1)}}\sum {\lambda _\alpha }({W_{ijij}}{S_{ii}^\alpha }{S_{jj}^\alpha } - {W_{ijji}}{S_{ii}^\alpha }{S_{jj}^\alpha })\\
   =&\sum {{\lambda _\alpha }{{\left| {{S^\alpha }W} \right|}^2}}  +  \frac{{{16(n-2)s^2}}}{{{n^2}{{(n - 1)}^2}}}\sum {{\lambda _\alpha }}
   + \frac{32s}{{n(n - 1)}}\sum {\lambda _\alpha }{W_{ijij}}{S_{ii}^\alpha }{S_{jj}^\alpha },\\
\end{aligned}
\end{equation}
where we choose the orthonormal eigenbasis $\left\{ {{e_i}} \right\}$ of ${S^\alpha }$.
On the other hand, we have
$${\omega _\alpha }{S_{ii}^\alpha } = {({\mathop W\limits^ \circ}({S^\alpha }))_{ii}} = \sum\limits_j {{W_{jiij}}{S_{jj}^\alpha }} ,$$
and
\begin{equation}
\sum\limits_{i,j} {{W_{ijij}}S_{ii}^\alpha {S_{jj}^\alpha }} =  - \sum\limits_i {{\omega _\alpha }{S}_{ii}^\alpha {S_{ii}^\alpha }}  =  - {\omega _\alpha }{\left| {{S^\alpha }} \right|^2} =  - {\omega _\alpha }.
\end{equation}
From (3.5) and (3.6), we get
\begin{equation*}
\begin{aligned}
\langle {{\mathop R\limits^ \circ}}  ({R^{S_0^2}}),{R^{S_0^2}}\rangle
  =&\sum {{\lambda _\alpha }{{\left| {{S^\alpha }W} \right|}^2}}+\frac{{16(n - 2){s^2}}}{{{n^2}{{(n - 1)}^2}}}\sum {{\lambda _\alpha }} - \frac{{32s}}{{n(n - 1)}}\sum {{\lambda _\alpha \omega _\alpha}}\\
  =&\sum {{\lambda _\alpha }{{\left| {{S^\alpha }W} \right|}^2}}+\frac{{16(n - 2){s^2}}}{{{n^2}{{(n - 1)}^2}}}\sum{{\lambda _\alpha }}\\
   & - \frac{{32s}}{{n(n - 1)}}\sum {{\lambda _\alpha^2}}+\frac{{32s^2}}{{n^2(n - 1)^2}}\sum {{\lambda _\alpha}}\\
  =&\sum {{\lambda _\alpha }{{\left| {{S^\alpha }W} \right|}^2}}+\frac{{16n{s^2}}}{{{n^2}{{(n - 1)}^2}}}\sum{{\lambda _\alpha }}- \frac{{32s}}{{n(n - 1)}}\sum {{\lambda _\alpha^2}}.
\end{aligned}
\end{equation*}
Combing the above with (3.4), we obtain from proposition \ref {pro3.3}
\begin{equation*}
\begin{aligned}
3\langle \Delta R,R\rangle = &\left\langle {{\mathop R\limits^ \circ} ({R^{S_0^2}}),{R^{S_0^2}}} \right\rangle  + \frac{{16(n - 4)}}{{{3n^2}}}s{\left(tr({(\mathop R\limits^ \circ)^2})+\frac{s^2}{n^2} \right)} + 16\sum {{\lambda _\alpha^3 }}  - \frac{{16}}{{{n^3}}}{s^3}\\
  =&\sum {{\lambda _\alpha }{{\left| {{S^\alpha }W} \right|}^2}}+8\left( {\frac{{n+2}}{{{n^2}{{(n - 1)}^2}}} + \frac{{2(n - 4)}}{{3{n^4}}} - \frac{2}{{{n^3}}}} \right){s^3}\\
   &+8\left( { - \frac{4}{{n(n - 1)}} + \frac{{2(n - 4)}}{{3{n^2}}}} \right)s\sum {\lambda _\alpha ^2}  + 16\sum {\lambda _\alpha ^3} \\
  =&\sum {{\lambda _\alpha }{{\left| {{S^\alpha }W} \right|}^2}}+8\left(\frac{-n^3+6n^2+12n-8}{3n^4(n-1)^2} \right){s^3}\\
   &+8\left(\frac{2n^2-22n+8}{3n^2(n-1)} \right)s\sum {\lambda _\alpha ^2}  + 16\sum {\lambda _\alpha ^3}.
\end{aligned}
\end{equation*}
The proof is completed.
\end{proof}

\begin{remark}
\label{re3.5}
If the curvature operator of the second kind is nonnegative, i.e., the ${\lambda _\alpha } \ge 0$ for all $\alpha $, according to the power mean inequality, we have
\begin{equation}
\begin{aligned}
\sum {\lambda _\alpha ^2}  \ge& \left( {\frac{2}{{(n - 1)(n + 2)}}} \right){\left( {\sum {{\lambda _\alpha }} } \right)^2}\\
  =& \left( {\frac{2}{{(n - 1)(n + 2)}}} \right){\left( {\frac{{n + 2}}{{2n}}s} \right)^2} \\
  =& \frac{{n + 2}}{{2{n^2}(n - 1)}}{s^2}\\
\end{aligned}
\end{equation}
and
\begin{equation}
\begin{aligned}
\sum {\lambda _\alpha ^3}  \ge& {\left( {\frac{2}{{(n - 1)(n + 2)}}} \right)^2}{\left( {\sum {{\lambda _\alpha }} } \right)^3}\\
  =&{\left( {\frac{2}{{(n - 1)(n + 2)}}} \right)^2}{\left( {\frac{{n + 2}}{{2n}}s} \right)^3}\\
  =& \frac{{n + 2}}{{2{n^3}{{(n - 1)}^2}}}{s^3}.\\
\end{aligned}
\end{equation}
with equalities hold if and only if $\lambda _\alpha = \lambda _\beta$ for all $\alpha, \beta$.

(3.3) is rewritten as
\begin{equation}
\label{3.9}
\begin{aligned}
3\langle \Delta R,R\rangle
  =&\sum {{\lambda _\alpha }{{\left| {{S^\alpha }W} \right|}^2}}+ \left( {\sum {\lambda _\alpha ^3}  - \left( \frac{2}{(n-1)(n+2)} \right)^2 {{\left( {\sum {{\lambda _\alpha }} } \right)}^3}} \right)\\
   &+\frac{{2\left( {{n^2} - 11n + 4} \right)}}{{3n\left( {n - 1} \right)\left( {n + 2} \right)}}\sum {{\lambda _\beta }} \left( {{{\sum {\lambda _\alpha ^2 - \frac{2}{(n-1)(n+2)} \left( {\sum {{\lambda _\alpha }} } \right)} }^2}} \right).\\
 \end{aligned}
\end{equation}
If $n\geq11$ and the curvature operator of the second kind is nonnegative, from (3.7), (3.8) and \eqref{3.9} we have
\begin{equation}
\label{1.1}
\frac{1}{2}\Delta {\left| R \right|^2} = {\left| {\nabla R} \right|^2} + \left\langle {\Delta R,R} \right\rangle\geq0.
\end{equation}
By using the maximum principle, from (3.10) we show that  Einstein manifolds of dimension $n\geq11$ with nonnegative  curvature operator of the second kind are constant curvature spaces.
\end{remark}

\section{Proof of main theorems}
In this section, we should prove that Einstein manifolds of dimension $n\geq4$ with nonnegative  curvature operator of the second kind are constant curvature spaces.  Under weak condition for the curvature operator of the second kind,  we estimate $\left\langle {\Delta R,R} \right\rangle$.

\begin{lemma}
\label{lem4.1}
For two sequences $\left\{ {{\lambda _i}} \right\}_{i = 1}^N$ and $\left\{ {{\theta _i}} \right\}_{i = 1}^N$ which satisfy ${\lambda _i} \le {\lambda _{i + 1}}$ and ${\theta _i} \ge 0$. Set $\Theta  = \max \left\{ {{\theta _i}} \right\}$ and $k = \left[ {\frac{{\sum {{\theta _i}} }}{\Theta }} \right]$. If $\sum\limits_{i = 1}^k {{\lambda _i}}  \ge 0$, then
$$\sum\limits_{i = 1}^N {{\lambda _i}{\theta _i}}  \ge 0.$$
\end{lemma}

\begin{remark}
Although Lemma 4.1 has been proved in \cite{{CMR},{NPW}}, for completeness,
we also write it out.
\end{remark}

\begin{proof}
We have
\begin{equation*}
\begin{aligned}
\sum\limits_{i = 1}^N {{\lambda _i}{\theta _i}}
 \ge& \sum\limits_{i = 1}^k {{\lambda _i}{\theta _i}}  + \sum\limits_{i = k + 1}^N {{\lambda _{k+1}}{\theta _i}}\\
   =& {\lambda _{k + 1}}\sum\limits_{i = 1}^N {{\theta _i}}  + \sum\limits_{i = 1}^k {({\lambda _i} - {\lambda _{k + 1}}){\theta _i}}\\
 \ge& {\lambda _{k + 1}}\sum\limits_{i = 1}^N {{\theta _i}}  + \Theta \sum\limits_{i = 1}^k {({\lambda _i} - {\lambda _{k + 1}})}\\
  =& {\lambda _{k + 1}}\left( {\sum\limits_{i = 1}^N {{\theta _i}}  - k\Theta } \right) + \Theta \sum\limits_{i = 1}^k {{\lambda _i}}\\
  \geq &0.
\end{aligned}
\end{equation*}
The proof is completed.
\end{proof}

\begin{lemma}
\label{lem4.2}
Let  $\left\{ {{\lambda _\alpha }} \right\}$ be the eigenvalues of the  curvature operator of the second kind ${\mathop R\limits^ \circ}$ on Riemannian manifolds of dimension $n\geq4$. If ${\mathop R\limits^ \circ}$ is $k_1=[\frac{n+2}{4}]$-nonnegative, then
$$\sum {{\lambda _\alpha }{{\left| {{S^\alpha }W} \right|}^2}}  \ge 0.$$
\end{lemma}

\begin{proof}
Set ${\mathop R\limits^ \circ} = \frac{1}{2}{\mathop {g \circledwedge g}\limits^ \circ}$ in \eqref{3.2}. Thus $\frac{1}{2}{\mathop {g \circledwedge g}\limits^ \circ}$ is identity on $S_0^2(V)$ and $R_{ijkl}=\frac{1}{2}(g \circledwedge g)_{ijkl}=g_{ik}g_{jl}-g_{il}g_{jk}$. From (2.2) and (3.2) we can get
\begin{equation}
\begin{aligned}
{\sum {\left| {{S^\alpha }W} \right|} ^2}
=& \left\langle {\frac{1}{2}\mathop {g \circledwedge g}\limits^ \circ  \left( {{W^{S_0^2}}} \right),{W^{S_0^2}}} \right\rangle \\
=& \frac{{2n + 32}}{n}\left( {n - 1} \right){g_{st}}{W_{sjkl}}{W_{tjkl}} - 5\left( {{g_{ts}}{g_{pq}} - {g_{tq}}{g_{ps}}} \right){W_{tpkl}}{W_{sqkl}} \\
 &+ 4\left( {{g_{tp}}{g_{sq}} - {g_{tq}}{g_{sp}}} \right){W_{tjpl}}{W_{sjql}} - \frac{{16}}{{{n^2}}}n(n-1){\left| W \right|^2}\\
=& \frac{{(2n + 32)(n-1)}}{n}{\left| W \right|^2} - 10{\left| W \right|^2} - 2{\left| W \right|^2} - \frac{{16(n-1)}}{{{n}}}{\left| W \right|^2}\\
=& \frac{{2\left( {{n^2} + n - 8} \right)}}{n}{\left| W \right|^2}.\\
\end{aligned}
\end{equation}

For a fixed ${\left| {{S^\alpha }W} \right|^2}$, we choose an orthonormal eigenbasis $\left\{ {{e_i}} \right\}$ of ${S^\alpha }$ with the corresponding eigenvalues  $\left\{ {{a_i}} \right\}$ satisfying $\sum {{a_i}}  = 0$ and $\sum {a_i^2}  = 1$. so
\begin{equation*}
{\left| {{S^\alpha }W} \right|^2}
 = \sum\limits_{i,j,k,l} {{{({S^\alpha }W)}_{ijkl}}{{({S^\alpha }W)}_{ijkl}}}  = \sum\limits_{i,j,k,l} {\left(\sum\limits_{i^{'}\in\{i,j,k,l\}}{a_{i^{'}}} \right)^2}{W_{ijkl}}^2.
\end{equation*}

Let's compute the global maximum of ${({a_i} + {a_j} + {a_k} + {a_l})^2}$.
By the algebraic curvature properties of $W$, there are three cases:

$(1)$ $i,j,k,l$ is not equal to each other;

$(2)$ $i=k$ and $i,j,l$  is not equal to each other;

$(3)$ $i=k,j=l$ and $i\ne j$.

Set
$$\varphi \left( {{a_1},{a_2}, \cdots ,{a_n},{\mu _1},{\mu _2}} \right) = \left( {{a_i} + {a_j} + {a_k} + {a_l}} \right) + {\mu _1}\sum\limits_{m = 1}^n {{a_m}}  + {\mu _2}\left( {\sum\limits_{m = 1}^n {a_m^2}  - 1} \right).$$
For $q \notin \left\{ {i,j,k,l} \right\},$  we have
\begin{equation}\frac{{\partial \varphi}}{{\partial {a_q}}} = {\mu _1} + 2{a_q}{\mu _2}=0.\end{equation}

In case $(1)$, for $p \in \left\{ {i,j,k,l} \right\},$  we have
\begin{equation}\frac{{\partial \varphi}}{{\partial {a_p}}} = 1 + {\mu _1} + 2{a_p}{\mu _2}=0.\end{equation}
From (4.2) and (4.3), we get
$$\left\{ {\begin{array}{*{20}{c}}
{\left( {n - 4} \right){a_q} + 4{a_i} = 0}\\
{\left( {n - 4} \right)a_q^2 + 4a_i^2 = 1}
\end{array}} \right. , $$
which implies $a_i^2 = \frac{{n - 4}}{{4n}}$. Thus the extremum of ${({a_i} + {a_j} + {a_k} + {a_l})^2}$ is $\frac{4({n - 4})}{{n}}$.

In case $(2)$, for $p \in \left\{ {j,l} \right\},$  we get
$$\left\{ {\begin{array}{*{20}{c}}
{\frac{{\partial \varphi}}{{\partial {a_i}}} = 2 + {\mu _1} + 2{a_i}{\mu _2}=0}\\
{\frac{{\partial \varphi}}{{\partial {a_p}}} = 1 + {\mu _1} + 2{a_p}{\mu _2}=0}
\end{array}} \right..  $$
Combing the above with (4.2), we have
$$\left\{ {\begin{array}{*{20}{c}}
{\left( {n - 3} \right){a_q} + {a_i}+2{a_j} = 0}\\
{{a_q} + {a_i}-2{a_j} = 0}\\
{\left( {n - 3} \right)a_q^2 + a_i^2+2a_j^2 = 1}
\end{array}} \right. , $$
which implies  $a_i^2 = \frac{{2(n - 2)^2}}{{n(3n-8)}}$ and $a_j^2=\frac{{(n - 4)^2}}{{2n(3n-8)}}$. Thus the extremum of ${({a_i} + {a_j} + {a_k} + {a_l})^2}$ is $\frac{2({3n - 8})}{{n}}$.

In case $(3)$, for $p \in \left\{ {i,j} \right\},$  we get
\begin{equation}
{\frac{{\partial \varphi}}{{\partial {a_p}}} = 2 + {\mu _1} + 2{a_p}{\mu _2}=0}.
\end{equation}
From (4.2) and (4.4), we get
$$\left\{ {\begin{array}{*{20}{c}}
{\left( {n - 2} \right){a_q} + 2{a_i} = 0}\\
{\left( {n - 2} \right)a_q^2 + 2a_i^2 = 1}
\end{array}} \right. , $$
which implies $a_i^2 = \frac{{n - 2}}{{2n}}$. Thus the extremum of ${({a_i} + {a_j} + {a_k} + {a_l})^2}$ is $\frac{8({n - 2})}{{n}}$.

For $\frac{8({n - 2})}{{n}} \ge \frac{2({3n - 8})}{{n}} \ge \frac{4({n - 4})}{{n}} $, we get from the above
\begin{equation}{\left| {{S^\alpha }W} \right|^2} \le \frac{{8(n - 2)}}{n}{\left| W \right|^2}.\end{equation}
It is apparent from (4.1) and (4.5) that
$$\frac{\sum {\left| {{S^\alpha }W} \right|} ^2}{\left| {{S^\alpha }W} \right|^2}\geq\frac{{2({n^2} + n - 8)}}{n}\frac{n}{{8(n - 2)}} = \frac{{{n^2} + n - 8}}{{4(n - 2)}} \ge \left[ {\frac{{n{\rm{ + }}2}}{4}} \right] = k_1.$$
By Lemma \ref{lem4.1}, we have
$$\sum {{\lambda _\alpha }{{\left| {{S^\alpha }W} \right|}^2}}  \ge 0.$$
The proof is completed.
\end{proof}

\subsection{Proof of Theorem \ref{th1.1}}
\
\

In order to prove Theorem \ref{th1.1},  we need to estimate the second and third items on the right of (3.9).  Now we give the following Lemmas 4.4 and 4.6.

\begin{lemma}
\label{lem4.3}
Let $\left\{ {{\lambda _i}} \right\}_{i = 1}^N$ be a nondecreasing sequence with $\sum\limits_{i = 1}^{{k_2}} {{\lambda _i}}=a \ge 0$ and $\sum\limits_{i = 1}^N {{\lambda _i}}  = C \ge 0$.  Then when $a=0$ or $a= \frac {k_2 C}{N}$,  $BC\sum\limits_{i = 1}^N {{\lambda _i}^2}  + \sum\limits_{i = 1}^N {{\lambda _i}^3}$ can achieve the global minimum,  where $B$ is a positive constant.
\end{lemma}	
\begin{proof}
Set
$$f= BC\sum\limits_{i = 1}^N {{\lambda _i}^2}  +\sum\limits_{i = 1}^N {\lambda _i ^3}=g+h,$$
where
$$g\left( {{\lambda _2}, \cdots ,{\lambda _{{k_2}}}} \right) = BC\left( {{{\left( {a - \sum\limits_{p = 2}^{{k_2}} {{\lambda _p}} } \right)}^2} + \sum\limits_{p = 2}^{{k_2}} {{\lambda _p}^2} } \right) + {\left( {a - \sum\limits_{p = 2}^{{k_2}} {{\lambda _p}} } \right)^3} + \sum\limits_{p = 2}^{{k_2}} {{\lambda _p}^3} $$
and
\begin{equation*}
\begin{aligned}
h\left( {{\lambda _{{k_2} + 1}}, \cdots ,{\lambda _{N - 1}}} \right)
=& BC\left( {\sum\limits_{q = {k_2} + 1}^{N - 1} {{\lambda _q}^2}  + {{\left( {(C-a) - \sum\limits_{q = {k_2} + 1}^{N - 1} {{\lambda _q}} } \right)}^2}} \right)\\
 &+ \sum\limits_{q = {k_2} + 1}^{N - 1} {{\lambda _q}^3}  + {\left( {(C-a) - \sum\limits_{q = {k_2} + 1}^{N - 1} {{\lambda _q}} } \right)^3}.
\end{aligned}
\end{equation*}
Thus $$\min f = \min g  + \min h.$$

First, we consider the value range of these variable $a$ and  $\{\lambda_m\}_{2}^{k_2}$. It is simple to show that  $0\leq a\leq\frac {k_2 C}{N},$ because $\left\{ {{\lambda _i}} \right\}_{i = 1}^N$ satisfy some conditions. Set $a_m= \left( {a - \left( {k_2 - m} \right)\frac{{C - a}}{{N - k_2}}} \right)$, we have
$${\lambda _m} \in \left[ { \frac {a_m}{m},\frac{{C - a}}{{N - k_2}}} \right],$$
In fact, if ${\lambda _m} > \frac{{C - a}}{{N - k_2}}$, then $\sum\limits_{j = k_2 + 1}^N {{\lambda _j}}  \ge \left( {N - k_2} \right){\lambda _m} > C - a$, which is a contradiction with $\sum\limits_{j = {k_2} + 1}^N {{\lambda _j}}= C - a$. And if ${{\lambda _m} < \frac{a_m}{m} = \frac{1}{m}\left( {a - \left( {k_2 - m} \right)\frac{{C - a}}{{N - k_2}}} \right)}$, then $\sum\limits_{i = 1}^m {{\lambda _i}}  \le m{\lambda _m} < \left( {a - \left( {k_2 - m} \right)\frac{{C - a}}{{N - k_2}}} \right)$, i.e., $- \sum\limits_{i = 1}^m {{\lambda _i}}  >- \left( {a - \left( {k_2 - m} \right)\frac{{C - a}}{{N - k_2}}} \right)$. By  $\sum\limits_{i = 1}^{k_2} {{\lambda _i}}=a$, we have $\sum\limits_{i = m + 1}^{k_2} {{\lambda _i}}  > \left( {k_2 - m} \right)\frac{{C - a}}{{N - k_2}}$, which is a contradiction with ${\lambda _{k_2}} \le \frac{{C - a}}{{N - k_2}}$. Set ${\Omega _{m}} = \prod\limits_{i = 2}^{m} {\left[ {\frac{{{a_i}}}{i},\frac{{C - a}}{{N - k_2}}} \right]} $ for $m=2, \cdots ,k_2$.

Since  ${{\lambda _{{k_2} + 1}}, \cdots ,{\lambda _{N - 1}}} $ are all nonnegative numbers, from (3.7) and (3.8) we get
\begin{equation} \min h = h\left( {\frac{{C - a}}{{N - {k_2}}}, \cdots ,\frac{{C - a}}{{N - {k_2}}}} \right).\end{equation}
We now turn to considering the function $f\left(a, {{\lambda _2}, \cdots ,{\lambda _{{k_2}}}},{{\lambda _{{k_2} + 1}}, \cdots ,{\lambda _{N - 1}}}  \right)$ restricted to ${\lambda _{{k_2} + 1}}=\cdots={\lambda _{N - 1}}=\frac{{C - a}}{{N - {k_2}}}$, where $a\in [0, \frac {k_2 C}{N}]$ and $\left( {{\lambda _2}, \cdots ,{\lambda _{{k_2}}}} \right) \in {\Omega _{k_2}}$. Thus we have
\begin{equation*}
\begin{aligned}
\frac{\partial f}{\partial a}
=& 2BC\left( {\left( {a - \sum\limits_{p = 2}^{{k_2}} {{\lambda _p}} } \right) - \left( {\left( {C - a} \right) - \sum\limits_{q = {k_2} + 1}^{N - 1} {{\lambda _q}} } \right)} \right) \\
&+ 3\left( {{{\left( {a - \sum\limits_{p = 2}^{{k_2}} {{\lambda _p}} } \right)}^2} - {{\left( {\left( {C - a} \right) - \sum\limits_{q = {k_2} + 1}^{N - 1} {{\lambda _q}} } \right)}^2}} \right)\\
 =& \left( 2a - C- \sum\limits_{p = 2}^{{k_2}} {{\lambda _p}}  + \sum\limits_{q = {k_2} + 1}^{N - 1} {{\lambda _q}}  \right)\left( {2BC + 3\left( {\lambda _1} + {\lambda _N} \right)} \right)\\
 =& 3\left( {{\lambda _1} - {\lambda _N}} \right)\left( {\frac{2}{3}BC + {\lambda _1} + {\lambda _N}} \right).\\
\end{aligned}
\end{equation*}
Thus  $\frac{\partial f}{\partial a}=0$ implies that ${\lambda _1}={\lambda _N}$ or ${\lambda _1} =  - \frac{2}{3}BC - {\lambda _N}$.

If ${\lambda _1}={\lambda _N}$, then $a=\frac {k_2 C}{N}$ and ${\lambda _1}= \dots ={\lambda _N}= \frac{C}{N}$.

 If ${\lambda _1} \ge  - \frac{2}{3}BC - {\lambda _N}$, the function $f$ increases with respective to $a$,  the minimum point of $f$ is $a=\frac {k_2 C}{N}$, and ${\lambda _1}= \dots ={\lambda _N}= \frac{C}{N}$.

If ${\lambda _1} <  - \frac{2}{3}BC - {\lambda _N}$, the function $f$ decreases with respective to $a$,  the minimum point of $f$ is $a=0$.
\end{proof}

\begin{remark}
Later, applying Lemma 4.4 to prove Theorem 1.1, we need to select $N=\frac{(n-1)(n+2)}{2}$ and $B=\frac{{2\left( {{n^2} - 11n + 4} \right)}}{{3n\left( {n - 1} \right)\left( {n + 2} \right)}}.$ For $n\geq11$, $B>0$. For $4\leq n\leq 10$, $B<0$. So (4.6) in the proof of Lemma 4.4 may not hold. However, for $8\leq n\leq 10$, Lemma 4.4 still holds. In fact,  we have
\begin{equation}
\begin{aligned}
h =& BC\left( {\sum\limits_{q = {k_2} + 1}^N {{\lambda _q}^2} } \right) + \left( {\sum\limits_{q = {k_2} + 1}^N {{\lambda _q}^3} } \right)\\
 \ge& \left( {B + \frac{1}{{N - {k_2}}}} \right)(C-a)\left( {\sum\limits_{q = {k_2} + 1}^N {{\lambda _q}^2} } \right)\\
 \ge& \left( {B + \frac{1}{N}} \right)(C-a)\left( {\sum\limits_{q = {k_2} + 1}^N {{\lambda _q}^2} } \right)\\
  =& \frac{{2\left( {{n^2} - 8n + 4} \right)}}{{3n\left( {n - 1} \right)\left( {n + 2} \right)}}(C-a)\left( {\sum\limits_{q = {k_2} + 1}^N {{\lambda _q}^2} } \right),
\end{aligned}
\end{equation}
where we use this inequality
$$\left( {N - {k_2}} \right)\left( {\sum\limits_{q = {k_2} + 1}^N {{\lambda _q}^3} } \right) \ge \left( {\sum\limits_{q = {k_2} + 1}^N {{\lambda _q}} } \right)\left( {\sum\limits_{q' = {k_2} + 1}^N {{\lambda _{q'}}^2} } \right).$$
Thus since
 $\frac{{2\left( {{n^2} - 8n + 4} \right)}}{{3n\left( {n - 1} \right)\left( {n + 2} \right)}} > 0$ for $8\leq n\leq 10$,  it is easy to see from (3.7) and (4.7) that $\min h = h \left( \frac{C}{N-k_2}, \cdots, \frac{C}{N-k_2} \right)$. The remainder of the argument is analogous to that in Lemma 4.4.
\end{remark}

\begin{lemma}
\label{lem4.4}
Let $\left\{ {{\lambda _i}} \right\}_{i = 1}^N$ be a nondecreasing sequence with $\sum\limits_{i = 1}^{{k_2}} {{\lambda _i}}=0$ and $\sum\limits_{i = 1}^N {{\lambda _i}}  = C\geq0$, where $N=\frac{(n-1)(n+2)}{2}$. Then the global minimum point of $BC\sum\limits_{i = 1}^N {{\lambda _i}^2}  + \sum\limits_{i = 1}^N {{\lambda _i}^3}$ may be
$${\lambda _1} =  \cdots  = {\lambda _{{k_2}}}=0,{\lambda _{{k_2} + 1}} =  \cdots  = {\lambda _N}= \frac{C}{N-k_2}$$
or
$${\lambda _1} =  - \left( {{k_2} - 1} \right)\frac{C}{{N - {k_2}}},{\lambda _{{k_2}}} = {\lambda _{{k_2} + 1}} =  \cdots  = {\lambda _N} = \frac{C}{{N - {k_2}}},$$ where $B=\frac{{2\left( {{n^2} - 11n + 4} \right)}}{{3n\left( {n - 1} \right)\left( {n + 2} \right)}}$ with $n\geq8$.
\begin{proof}
By Lemma \ref{lem4.3} and Remark 4.5,   we only consider this case $a=0$. In order to solving the global minimum point of $BC\sum\limits_{i = 1}^N {{\lambda _i}^2}  + \sum\limits_{i = 1}^N {{\lambda _i}^3}$, from the proof of  Lemma \ref{lem4.3} and Remark 4.5 we need to consider the function
$$ g = BC\left( {{{\left( { - \sum\limits_{i = 2}^{{k_2}} {{\lambda _i}} } \right)}^2} + \sum\limits_{i = 2}^{{k_2}} {{\lambda _i}^2} } \right) + {\left( { - \sum\limits_{i = 2}^{{k_2}} {{\lambda _i}} } \right)^3} + \sum\limits_{i = 2}^{{k_2}} {{\lambda _i}^3}, $$
where $(\lambda_2, \cdots ,\lambda_{k_2})\in \Omega_{k_2}$.

For $p \in \left\{ {2,\cdots, k_2} \right\},$ we get
\begin{equation*}
\begin{aligned}
\frac{{\partial g}}{{\partial {\lambda _p }}}
=& 2BC \left(\lambda _p  + {  \sum\limits_{i  = 2}^{k_2} {\lambda _i } }\right) +3{\lambda _p^2 } - 3{\left( { \sum\limits_{i  = 2}^{k_2} {{\lambda _i }} } \right)^2}\\
=& 3\left( {{\lambda _p} - {\lambda _1}} \right)\left( {\frac{2}{3}BC + {\lambda _p} + {\lambda _1}} \right).
\end{aligned}
\end{equation*}
Thus $\frac{{\partial g_{k_2}}}{{\partial {\lambda _p }}} =0$ implies that ${\lambda _p}={\lambda _1}$ or ${\lambda _p}= -\frac{2}{3}BC -{\lambda _1} $. If ${\lambda _p}= -\frac{2}{3}BC -{\lambda _1} $, then ${\lambda _p}= -\frac{2}{3}BC -{\lambda _1}> {\lambda _N}$ for ${\lambda _1} <  - \frac{2}{3}uC - {\lambda _N}$ (see the proof of Lemma \ref{lem4.3}). This is a contradiction with $\left\{ {{\lambda _i}} \right\}_{i = 1}^N$ be a nondecreasing sequence. So ${\lambda _p}={\lambda _1}$ can only hold. Hence the local minimum point of $g$  is
$ (0, \dots, 0).$

 Now we consider the function $g$ on $ \partial \Omega_{k_2}.$ A point $x=(\lambda_2, \cdots \lambda_{k_2}) \in \partial \Omega_{k_2} $ has at least one component $\lambda_{m+1}$ satisfying $\lambda_{m+1} = \frac{a_{m+1}}{m+1}=\frac{\left( {(m+1)-k_2 } \right)C}{(m+1){(N - k_2)}}$ or $\lambda_{m+1} = \frac{{C}}{{N - k_2}}$. If $\lambda_{m+1} = \frac{a_{m+1}}{m+1}$, then $\sum\limits_{i = 1}^{m + 1} {{\lambda _i}}  \le \left( {m + 1} \right){\lambda _{m + 1}} = {a_{m + 1}}$, and $\sum\limits_{i = m + 2}^{k_2} {{\lambda _i}}  \ge  - {a_{m + 1}} = \frac{{C\left( {k_2 - \left( {m + 1} \right)} \right)  }}{{N - k_2}}$. Hence ${\lambda _{m + 2}} =  \cdots  = {\lambda _{k_2}} = \frac{{C }}{{N - k_2}}$ and $\lambda_2= \cdots =\lambda_{m+1} = \frac{a_{m+1}}{m+1}$. If $\lambda_{m+1} = \frac{{C }}{{N - k_2}}$, then $\lambda_{m+1} = \cdots = \lambda_{k_2}= \frac{{C }}{{N - k_2}}$ and $\sum\limits_{i = 1}^{m} {{\lambda _i}}  = { - \left( {k_2 - m} \right)\frac{{C - a}}{{N - k_2}}} =a_m$. Hence we set
$$ g_{m}(x) = BC\left( {{{\left( { a_m - \sum\limits_{i = 2}^{{m}} {{\lambda _i}} } \right)}^2} + \sum\limits_{i = 2}^{{m}} {{\lambda _i}^2} } \right) + {\left( {a_m - \sum\limits_{i = 2}^{{m}} {{\lambda _i}} } \right)^3} + \sum\limits_{i = 2}^{{m}} {{\lambda _i}^3}.$$
For $p \in \left\{ {2,\cdots, m} \right\},$ we have
\begin{equation*}
\begin{aligned}
\frac{{\partial g_{m}}}{{\partial {\lambda _p }}}
=& 2BC \left({\lambda _p } - {\left( {a_m - \sum\limits_{i  = 2}^{m} {{\lambda _i }} } \right)}\right) +3{\lambda _p^2 } - 3{\left( { a_m - \sum\limits_{i  = 2}^{m} {{\lambda _i }} } \right)^2}\\
=& 3\left( {{\lambda _p} - {\lambda _1}} \right)\left( {\frac{2}{3}BC + {\lambda _p} + {\lambda _1}} \right).
\end{aligned}
\end{equation*}
An argument similar to the one used the above that the  local minimum point of $g_{m}(x)$ in the $ \Omega _m$ is $ (\frac{a_m}{m}, \cdots ,\frac{a_m}{m})$. Similarly, we consider the function $g_m$ on $ \partial \Omega_{m}.$ Repeating the above process, we can obtain that there are
 all minimum points of $g(x)$ as follows:
$$x_{k_2}=(0, \cdots 0),$$
$${x_{k_2-1}} = \left( {\frac{{{a_{k_2 - 1}}}}{{k_2 - 1}}, \cdots ,\frac{{{a_{k_2 - 1}}}}{{k_2 - 1}},\frac{{C }}{{N - k_2}}} \right),$$

$$\cdots$$

$${x_m} = \left( {\overbrace {\frac{{{a_m}}}{m}, \cdots ,\frac{{{a_m}}}{m}}^{m-1},\overbrace {\frac{{C}}{{N - k_2}}, \cdots ,\frac{{C}}{{N - k_2}}}^{k_2 - m}} \right),$$

$$\cdots,$$

$${x_1} = \left( {\frac{{C}}{{N - k_2}}, \cdots ,\frac{{C}}{{N - k_2}}} \right).$$
Thus
\begin{equation*}
\begin{aligned}
{g}\left( {{x_m}} \right)
 =& BC\left( m{\left( {\frac{{{a_m}}}{m}} \right)^2} + \left( {k_2 - m} \right){\left( {\frac{{C}}{{N - k_2}}} \right)^2}\right) + m{\left( {\frac{{{a_m}}}{m}} \right)^3} + \left( {k_2 - m} \right){\left( {\frac{{C}}{{N - k_2}}} \right)^3}\\
 =&BC\left( \frac{ ( {k_2 - m} )^2{C^2 }}{m{(N - k_2)^2}}  +  {\frac{(k_2 - m){C^2 }}{{(N - k_2)^2}}} \right)- { {\frac{(k_2 - m)^3{C^3 }}{m^2{(N - k_2)^3}}} } + {\frac{( {k_2 - m}){C^3 }}{{(N - k_2)^3}}}.
\end{aligned}
\end{equation*}
Now we consider $g$ as a function of $m$. So we get
\begin{equation*}
\begin{aligned}
g'\left( m \right)
=&  - \frac{B{k_2^2{C^3}}}{{{m^2}{{\left( {N - {k_2}} \right)}^2}}} + \frac{(2k_2-3m){k_2^2{C^3}}}{{{m^3}{{( {N - {k_2}})}^3}}}\\
=&  - \frac{{k_2^2{C^3}}}{{{m^3}{{\left( {N - {k_2}} \right)}^3}}}\left( {\left( {B\left( {N - {k_2}} \right) + 3} \right)m - 2{k_2}} \right).
\end{aligned}
\end{equation*}
From the above, it is easy to see that
the minimum point of $g(m)$ may be $1$ or $k_2$.
\end{proof}
\end{lemma}

\begin{proof}[{\bf Proof of Theorem
1.1}]
By using (3.9), we prove Theorem 1.1 in two cases: $C=0$ and $C>0$, where $C=\sum \lambda_\beta$.

If $C=0$, i.e., $s=0$, then $\left( {M,g} \right)$ is flat.

Now let's prove the second case, i.e., $C>0$.

Let $F$ denote sum of the second and third items on the right of (3.9).

First, for $n\geq8$, we prove this theorem.

By Remark 4.5, Lemmas 4.4 and 4.6 with $N = \frac{(n-1)(n+2)}{2}$ and $B=\frac{{2\left( {{n^2} - 11n + 4} \right)}}{{3n\left( {n - 1} \right)\left( {n + 2} \right)}},$ we obtain that the global minimum point of $f=BC\sum\limits_{i = 1}^N {{\lambda _i}^2}  + \sum\limits_{i = 1}^N {{\lambda _i}^3}$ may be
$\left(\frac{C}{N}, \cdots, \frac{C}{N}\right)$, $\left( {\overbrace {0, \cdots ,0}^{k_2},\overbrace {\frac{{C}}{{N - k_2}}, \cdots ,\frac{{C}}{{N - k_2}}}^{N - k_2}} \right)$ or $\left(-\frac{(k_2-1)C}{N-k_2}, \frac{C}{N-k_2},\cdots, \frac{C}{N-k_2}\right).$

When the  minimum point of $f$ is $\left(\frac{C}{N}, \cdots, \frac{C}{N}\right)$, it is clear form (3.7) and (3.8) that $F=0$.

When the  minimum point of $f$ is $\left( {\overbrace {0, \cdots ,0}^{k_2},\overbrace {\frac{{C}}{{N - k_2}}, \cdots ,\frac{{C}}{{N - k_2}}}^{N - k_2}} \right)$, by direct calculation, we get
$$F=\frac{{{k_2}}C^3}{{N^2{{ ({N - {k_2}}) }}}}\left(\frac {2N - {k_2}} {N - {k_2}}+ \frac{n^2-11n+4}{3n} \right).$$
According to the assumption of Theorem 1.1, from the above it can be known that $F>0$.

When the  minimum point of $f$ is $\left(-\frac{(k_2-1)C}{N-k_2}, \frac{C}{N-k_2},\cdots, \frac{C}{N-k_2}\right),$  we have
\begin{equation*}
\begin{aligned}
F= & \frac{{2\left( {{n^2} - 11n + 4} \right)}}{{3n\left( {n - 1} \right)\left( {n + 2} \right)}}C^3\left( {{{\left( { \frac{{{k_2} - 1}}{{N - {k_2}}}} \right)}^2} + \left( {N - 1} \right){{\left( {\frac{{1}}{{N - {k_2}}}} \right)}^2}}- \frac{{{1}}}{{{N}}} \right) \\
 &+ C^3\left( {{{-\left( {  \frac{{{k_2} - 1}}{{N - {k_2}}}} \right)}^3} + \left( {N - 1} \right){{\left( {\frac{{1}}{{N - {k_2}}}} \right)}^3}} - \frac{{{1}}}{{{N^2}}} \right)\\
 =& \left( {\frac{{2\left( {{n^2} - 11n + 4} \right)}}{{3n\left( {n - 1} \right)\left( {n + 2} \right)}}\frac{{\left( {N - 1} \right){k_2}^2}}{{N{{\left( {N - {k_2}} \right)}^2}}} + \frac{{{k_2}^2\left( { - {N^2}{k_2} + 3{N^2} - 3N + {k_2}} \right)}}{{{N^2}{{\left( {N - {k_2}} \right)}^3}}}} \right){C^3}\\
=& \frac{{\left( {N - 1} \right){k_2}^2}C^3}{{N{{\left( {N - {k_2}} \right)}^2}}}\left( {\frac{{2\left( {{n^2} - 11n + 4} \right)}}{{3n\left( {n - 1} \right)\left( {n + 2} \right)}} + \frac{{ - \left( {N + 1} \right){k_2} + 3N}}{{N\left( {N - {k_2}} \right)}}} \right)\\
=& \frac{{2\left( {N - 1} \right){k_2}^2}C^3}{{N{{\left( {N - {k_2}} \right)}^2}}}\frac{{\left( {{n^2} - 11n + 4} \right)\left( {N - {k_2}} \right) + 3n\left( {3N - \left( {N + 1} \right){k_2}} \right)}}{{3n\left( {n - 1} \right)\left( {n + 2} \right)\left( {N - {k_2}} \right)}}.
\end{aligned}
\end{equation*}
According to the assumption of Theorem 1.1, from the above it can be known that $F> 0$.

In a word, combing the above with Lemma 4.3, from (3.9)  we have
\begin{equation*}
\label{1.1}
\frac{1}{2}\Delta {\left| R \right|^2} = {\left| {\nabla R} \right|^2} + \left\langle {\Delta R,R} \right\rangle\geq0.
\end{equation*}
Thus we have
$$\nabla R=0$$
and
$$\left\langle {\Delta R,R} \right\rangle=0.$$
It can be seen from the above proof that $\left\langle {\Delta R,R} \right\rangle=0$ holds if and only if
 ${\lambda _1} =  \cdots  = {\lambda _N} = \frac{C}{N}$.
Hence the curvature operator of the second kind ${\mathop R\limits^ \circ}$ is a constant multiples of $\frac{1}{2}{\mathop {g\circledwedge g}\limits^ \circ}.$

Second, we prove this theorem for $6\leq n\leq7.$

Since the curvature operator of the second kind ${\mathop R\limits^ \circ}$ is nonnegive, i.e. $\lambda_1 = a \ge0$, the function $f\left( {{\lambda _1}, \cdots ,{\lambda _N}} \right)$ in Lemma \ref{lem4.3} becomes
\begin{equation*}
\begin{aligned}
f\left( {a,{\lambda _2}, \cdots ,{\lambda _N}} \right)
=& BC\left( {{a^2} + \sum\limits_{i = 2}^N {{\lambda _i}^2} } \right) + \left( {{a^3} + \sum\limits_{i = 2}^N {{\lambda _i}^3} } \right)\\
=& BC\left( {{a^2} + \sum\limits_{i = 2}^{N - 1} {{\lambda _i}^2}  + {{\left( {C - a - \sum\limits_{i = 2}^{N - 1} {{\lambda _i}} } \right)}^2}} \right) \\
 &+ \left( {{a^3} + \sum\limits_{i = 2}^{N - 1} {{\lambda _i}^3}  + {{\left( {C - a - \sum\limits_{i = 2}^{N - 1} {{\lambda _i}} } \right)}^3}} \right),\\
\end{aligned}
\end{equation*}
where $a \in \left[ {0,\frac{C}{N}} \right]$ and ${\lambda _N} \in \left[ {\frac{C}{N},C} \right]$.
\begin{equation*}
\begin{aligned}
\frac{{\partial f}}{{\partial a}}
=& BC\left( {2a - 2\left( {C - a - \sum\limits_{i = 2}^{N - 1} {{\lambda _i}} } \right)} \right) + \left( {3{a^2} - 3{{\left( {C - a - \sum\limits_{i = 2}^{N - 1} {{\lambda _i}} } \right)}^2}} \right)\\
=& 3\left( {a - {\lambda _N}} \right)\left( {\frac{2}{3}BC + a + {\lambda _N}} \right),\\
\end{aligned}
\end{equation*}
Thus $\frac{{\partial f}}{{\partial a}} = 0 $ implies that ${a} =  - \frac{2}{3}BC - {\lambda _N}$ or ${a} = {\lambda _N}$. It is simple to prove that ${\lambda _N} >- \frac{2}{3}BC - {\lambda _N}$ for ${\lambda _N} \in \left[ {\frac{C}{N},C} \right]$. For $6 \leq n \leq 7$, we compute
\begin{equation*}
\begin{aligned}
 -({\lambda _N} + \frac{2}{3}BC)
 \leq&-\left( {\frac{1}{N} + \frac{2}{3}B} \right)C\\
=& -\frac{{4\left( {{n^2} - \frac{13}{2}n + 4} \right)}}{{9n\left( {n - 1} \right)\left( {n + 2} \right)}}C\\
\leq& 0.
\end{aligned}
\end{equation*}
From the above, it is easy to see that
the minimum point of $f$ satisfies ${\lambda _1}={\lambda _N}$.
 Hence the  minimum point of $f$ is $\left(\frac{C}{N}, \cdots, \frac{C}{N}\right)$, and $F=0$.
 The remainder of the argument is analogous to the front, which completes the proof of Theorem 1.1.
\end{proof}

\subsection{Proof of Theorem 1.2}

In an Einstein manifold, for ${\left| R \right|^2} = {\left| W \right|^2} + \frac{{2{s^2}}}{{n(n - 1)}}$, we have
\begin{equation}
\begin{aligned}
\alpha
=& {R_{tpsq}}{R_{tpkl}}{R_{sqkl}}\\
=& \left( {{W_{tpsq}} + \frac{s}{{n(n - 1)}}\left( {{g_{ts}}{g_{pq}} - {g_{tq}}{g_{ps}}} \right)} \right)\left( {{W_{tpkl}} + \frac{s}{{n(n - 1)}}\left( {{g_{tk}}{g_{pl}} - {g_{tl}}{g_{pk}}} \right)} \right)\\
 &\left( {{W_{sqkl}} + \frac{s}{{n(n - 1)}}\left( {{g_{sk}}{g_{ql}} - {g_{sl}}{g_{qk}}} \right)} \right)\\
=& {W_{tpsq}}{W_{tpkl}}{W_{sqkl}} + \frac{{3s}}{{n(n - 1)}}{W_{tpsq}}{W_{tpkl}}\left( {{g_{sk}}{g_{ql}} - {g_{sl}}{g_{qk}}} \right)\\
 &+ \frac{{{s^3}}}{{{n^3}{{(n - 1)}^3}}}\left( {{g_{ts}}{g_{pq}} - {g_{tq}}{g_{ps}}} \right)\left( {{g_{tk}}{g_{pl}} - {g_{tl}}{g_{pk}}} \right)\left( {{g_{sk}}{g_{ql}} - {g_{sl}}{g_{qk}}} \right)\\
=& {W_{tpsq}}{W_{tpkl}}{W_{sqkl}} + \frac{{6s}}{{n(n - 1)}}{\left| W \right|^2} + \frac{{4{s^3}}}{{{n^2}{{(n - 1)}^2}}}\\
=& {W_{tpsq}}{W_{tpkl}}{W_{sqkl}} + \frac{{6s}}{{n(n - 1)}}{\left| R \right|^2} - \frac{{8{s^3}}}{{{n^2}{{(n - 1)}^2}}}\\
\end{aligned}
\end{equation}
and
\begin{equation}
\begin{aligned}
\beta
=& {R_{tspq}}{R_{tjpl}}{R_{sjql}}\\
=& \left( {{W_{tspq}} + \frac{s}{{n(n - 1)}}\left( {{g_{tp}}{g_{sq}} - {g_{tq}}{g_{sp}}} \right)} \right)\left( {{W_{tjpl}} + \frac{s}{{n(n - 1)}}\left( {{g_{tp}}{g_{jl}} - {g_{tl}}{g_{jp}}} \right)} \right)\\
 &\left( {{W_{sjql}} + \frac{s}{{n(n - 1)}}\left( {{g_{sq}}{g_{jl}} - {g_{sl}}{g_{jq}}} \right)} \right)\\
=& {W_{tspq}}{W_{tjpl}}{W_{sjql}} + \frac{{3s}}{{n(n - 1)}}{W_{tspq}}{W_{tjpl}}\left( {{g_{sq}}{g_{jl}} - {g_{sl}}{g_{jq}}} \right)\\
 &+ \frac{{{s^3}}}{{{n^3}{{(n - 1)}^3}}}\left( {{g_{tp}}{g_{sq}} - {g_{tq}}{g_{sp}}} \right)\left( {{g_{tp}}{g_{jl}} - {g_{tl}}{g_{jp}}} \right)\left( {{g_{sq}}{g_{jl}} - {g_{sl}}{g_{jq}}} \right)\\
=& {W_{tspq}}{W_{tjpl}}{W_{sjql}} + \frac{{3s}}{{2n(n - 1)}}{\left| W \right|^2} + \frac{{(n - 2){s^3}}}{{{n^2}{{(n - 1)}^2}}}\\
=& {W_{tspq}}{W_{tjpl}}{W_{sjql}} + \frac{{3s}}{{2n(n - 1)}}{\left| R \right|^2} + \frac{{(n - 5){s^3}}}{{{n^2}{{(n - 1)}^2}}}.\\
\end{aligned}
\end{equation}
Due to Jack and Parker's result \cite{JP} that  for $n \le 5$,
$${W_{tpsq}}{W_{tpkl}}{W_{sqkl}} = 2 {W_{tspq}}{W_{tjpl}}{W_{sjql}},$$
Combing the above with (4.8) and (4.9), we get
\begin{equation}\alpha  = 2\beta  + \frac{{3s}}{{n(n - 1)}}{\left| R \right|^2} - \frac{{2{s^3}}}{{{n^2}(n - 1)}}.\end{equation}
It follows from (3.1) and (4.10) that
\begin{equation}\langle \Delta R,R\rangle  = 2{R_{lt}}{R_{ijkl}}{R_{ijkt}} - 6\beta  - \frac{{3s}}{{n(n - 1)}}{\left| R \right|^2} + \frac{{2{s^3}}}{{{n^2}(n - 1)}}.\end{equation}
By Lemma \ref{3.3} and (4.10), we get
\begin{equation}8\sum {{\lambda _\alpha }^3}  =  - 6\beta  + \frac{{3s}}{{n(n - 1)}}{\left| R \right|^2} - \frac{{2{s^3}}}{{{n^2}(n - 1)}} + \frac{{8{s^3}}}{{{n^3}}}.\end{equation}
Inserting (4.12) into (4.11) gives
\begin{equation*}
\begin{aligned}
\langle \Delta R,R\rangle
=& 2{R_{lt}}{R_{ijkl}}{R_{ijkt}}{\rm{ + }}8\sum {{\lambda _\alpha }^3}  - \frac{{6s}}{{n(n - 1)}}{\left| R \right|^2} - \frac{{4(n - 2){s^3}}}{{{n^3}(n - 1)}}\\
=& \frac{{8(n - 4)s}}{{3n(n - 1)}}\sum {{\lambda _\alpha }^2}  + 8\sum {{\lambda _\alpha }^3}  - \frac{{4(n + 2)}}{{3{n^3}(n - 1)}}{s^3}\\
=& \frac{{16(n - 4)}}{{3(n - 1)(n + 2)}}\sum {{\lambda _\alpha }} \left( {\sum {{\lambda _\alpha }^2}  - \frac{1}{N}{{\left( {\sum {{\lambda _\alpha }} } \right)}^2}} \right)\\
 &+ 8\left( {\sum {{\lambda _\alpha }^3}  - \frac{1}{{{N^2}}}{{\left( {\sum {{\lambda _\alpha }} } \right)}^3}} \right).\\
\end{aligned}
\end{equation*}
The remainder of the argument is analogous to that in Theorem 1.1, which completes the proof of Theorem 1.2.

\bibliographystyle{amsplain}

\end{document}